\documentclass[graybox]{svmult}

\usepackage{makeidx}        
\usepackage{graphicx}       
\usepackage{multicol}      
\usepackage[bottom]{footmisc}
\usepackage{newtxtext}      
\usepackage{newtxmath} 
\usepackage{amssymb,amsmath}

\allowdisplaybreaks

\makeindex             

\begin{document}

\title*{Time-Fractional Optimal Control\\ 
of Initial Value Problems on Time Scales\thanks{This is a preprint 
of a paper accepted for publication as a book chapter with 
Springer International Publishing AG. Submitted 23/Jan/2019; 
revised 27-March-2019; accepted 12-April-2019.}}

\author{Gaber M. Bahaa and Delfim F. M. Torres}

\authorrunning{G. M. Bahaa and D. F. M. Torres}

\institute{Gaber M. Bahaa 
\at Department of Mathematics and Computer Science, 
Faculty of Science, Beni-Suef University, Beni-Suef, Egypt;
and Department of Mathematics, Faculty of Science, Taibah University,
Al-Madinah Al-Munawarah, Saudi Arabia;
\email{Bahaa$_{-}$gm@yahoo.com}
\and 
Delfim F. M. Torres 
\at 
Center for Research and Development in Mathematics and Applications (CIDMA),
Department of Mathematics, University of Aveiro, 3810-193 Aveiro, Portugal;
\email{delfim@ua.pt}}

\maketitle

% ----------------------------------------------

\abstract{We investigate Optimal Control Problems (OCP) 
for fractional systems involving fractional-time derivatives 
on time scales. The fractional-time derivatives and integrals
are considered, on time scales, in the Riemann--Liouville sense. 
By using the Banach fixed point theorem, sufficient conditions for
existence and uniqueness of solution to initial value problems
described by fractional order differential equations on time scales
are known. Here we consider a fractional OCP with a performance 
index given as a delta-integral function of both state 
and control variables, with time evolving on an arbitrarily 
given time scale. Interpreting the Euler--Lagrange 
first order optimality condition with an adjoint problem, 
defined by means of right Riemann--Liouville fractional 
delta derivatives, we obtain an optimality system for the 
considered fractional OCP. For that, we first prove new
fractional integration by parts formulas on time scales.\\[0.3cm]
}

% ----------------------------------------------

\section{Introduction}

Let $\mathbb{T}$ be a time scale, that is,
a nonempty closed subset of $\mathbb{R}$.
We consider the following initial value problem:
\begin{equation}
\label{eq1}
\begin{gathered}
{_{t_{0}}^{\mathbb{T}}D}_{t}^{\alpha} y(t)=f(t,y(t)),
\quad t\in[t_{0}, t_{0}+a]=\mathcal{J}\subseteq\mathbb{T},
\quad 0< \alpha <1,\\
{_{t_{0}}^{\mathbb{T}}I}_{t}^{1-\alpha}y(t_{0})= 0,
\end{gathered}
\end{equation}
where ${_{t_{0}}^{\mathbb{T}}D}_{t}^{\alpha}$ is the (left)
Riemann--Liouville fractional derivative operator or order $\alpha$
defined on $\mathbb{T}$ and ${_{t_{0}}^{\mathbb{T}}I}_{t}^{1-\alpha}$
is the (left) Riemann--Liouville fractional integral operator or order
$1-\alpha$ defined on $\mathbb{T}$, as introduced in \cite{MyID:328}
(see also \cite{MyID:365,MyID:410}), and function $f:\mathcal{J}
\times\mathbb{R}\rightarrow \mathbb{R}$ is a right-dense continuous
function. Necessary and sufficient conditions for the existence
and uniqueness of solution to problem \eqref{eq1} are
already discussed in \cite{MyID:328}. Here, our goal 
is to prove optimality conditions for such systems.

Fractional Calculus (FC) is a generalization of classical calculus. 
It has been reported in the literature that systems described using
fractional derivatives give a more realistic behavior. There exists
many definitions of a fractional derivative. Commonly used
fractional derivatives are the classical Riemann--Liouville and
Caputo derivatives on continuous time scales. Fractional derivatives and
integrals of Riemann--Liouville and Caputo types have a vast
number of applications, across many fields of science and
engineering. For example, they can be used to model controllability,
viscoelastic flows, chaotic systems, Stokes problems,
thermo-elasticity, several vibration and diffusion processes,
bioengineering problems, and many other complex phenomena: see, e.g.,
\cite{Agra.1,Bah.7} and references therein. 

Fractional optimal control problems 
on a continuous time scale have attracted several authors in
the last two decades, and many techniques have been developed for
solving such problems, involving classical fractional
derivatives. Agrawal \cite{Agra.1,Agra.2} presented a general
formulation and proposed a numerical method to solve such problems.
In those papers, the fractional derivative was defined in the
Riemann--Liouville sense and the formulation was obtained by means
of a fractional variational principle and the Lagrange multiplier
technique. Using new techniques, Frederico and Torres 
\cite{MR2338631,Fre.1} obtained Noether-like theorems 
for fractional optimal control problems in both 
Riemann--Liouville and Caputo senses. In \cite{Mop.1,Mop.2}, 
Mophou and N'Gu\'{e}r\'{e}kata studied the fractional 
optimal control of diffusion equations involving
the classical Riemann--Liouville derivatives. In \cite{OZd.1},
Ozdemir investigated the fractional optimal control problem of a
distributed system in cylindrical coordinates whose dynamics are
defined in the classical Riemann--Liouville sense. For the state
of the art and many generalizations, see the recent books
\cite{MR3822307,MR3331286}.

The theory of fractional differential equations, specifically the
question of existence and uniqueness of solutions, is a research
topic of great importance \cite{ABN,E.Bajlekova,E.Hernandez}.
Another important area of study is dynamic equations on time scales,
which goes back to 1988 and the work of Aulbach and Hilger, and has
been used with success to unify differential and difference
equations \cite{ABRP,MR1062633,BP}.
Starting with a linear dynamic equation, Bastos et al. have
introduced the notion of fractional-order derivative on time scales,
involving time-scale analogues of Riemann--Liouville operators
\cite{BastosPhD,MR2728463,MyID:179}. Another approach originate from
the inverse Laplace transform on time scales \cite{MR2800417}. After
such pioneer work, the study of fractional calculus on time scales
developed in a popular research subject: see
\cite{MyID:330,BBT,MyID:320,MyID:324,MyID:358} and 
the more recent references \cite{MyID:413,MyID:379,MyID:422,MyID:403,MyID:425}.

To the best of our knowledge, the study of fractional optimal control problems 
for dynamical systems on time scales is under-developed, at least  
when compared to the continuous and discrete cases \cite{MR3718404,MR3726478}. 
Motivated by this fact, in this paper an Optimal Control Problem (OCP) for fractional 
initial value systems involving fractional-time
derivatives on time scales is considered. The fractional-time
derivative and integral are considered in the Riemann--Liouville
sense on time scales, as introduced in \cite{MyID:328}. 
We prove necessary optimality conditions for such OCPs.
The performance index of the Fractional Optimal Control Problem (FOCP) 
is considered as a non-autonomous delta integral 
of a function depending on state and control variables, 
and where the dynamic control system is expressed by a delta-differential system. 
Interpreting the Euler--Lagrange first order optimality condition 
with an adjoint problem, defined by means of the time-scale 
right fractional derivative in the sense of Riemann--Liouville, 
we obtain an optimality system for the FOCP on time scales. 

% ------------------------

\section{Preliminaries}
\label{sec:prelim}

In this section, we collect notations, definitions, and results,
which are needed in the sequel. We use
$\mathcal{C}(\mathcal{J},\mathbb{R})$ for the Banach space of
continuous functions $y$ with the norm
$\|y\|_{\infty}=\sup\left\{|y(t)| : t\in\mathcal{J}\right\}$, where
$\mathcal{J}$ is a time-scale interval. 

% ------------------------------------

\subsection{Time-scale essentials}

A time scale $ \mathbb{T}$ is an
arbitrary nonempty closed subset of $ \mathbb{R}$. The reader
interested on the calculus on time scales is referred to the books
\cite{BP,BP1}. For a survey, see \cite{ABRP}. Any time scale
$\mathbb{T}$ is a complete metric space with the distance
$d(t,s)=|t-s|$, $t,s\in\mathbb{T}$. Consequently, according to the
well-known theory of general metric spaces, we have for $\mathbb{T}$
the fundamental concepts such as open balls (intervals),
neighborhoods of points, open sets, closed sets, compact sets, etc.
In particular, for a given number $\delta>0$, the
$\delta$-neighborhood $\mathrm{U}_{\delta}(t)$ of a given point
$t\in\mathbb{T}$ is the set of all points $s\in\mathbb{T}$ such that
$d(t,s)< \delta$. We also have, for functions
$f:\mathbb{T}\rightarrow\mathbb{R}$, the concepts of limit,
continuity, and the properties of continuous functions on a general
complete metric space. Roughly speaking, the calculus on time scales
begins by introducing and investigating the concept of derivative
for functions $f :\mathbb{T}\rightarrow\mathbb{R}$. In the
definition of derivative, an important role is played by the
so-called jump operators. 

\begin{definition}
\label{def:jump:oper} (See \cite{BP1}). 
Let $\mathbb{T}$ be a time scale. 
For $t \in \mathbb{T}$, we define the forward jump
operator $\sigma:\mathbb{T}\rightarrow \mathbb{T}$ by
$\sigma(t):=\inf\{s \in\mathbb{T} : s > t\}$, and the backward jump
operator $\rho:\mathbb{T}\rightarrow \mathbb{T}$ by
$\rho(t):=\sup\{s \in\mathbb{T} : s < t\}$.
\end{definition}

\begin{remark}
In Definition~\ref{def:jump:oper}, we put $\inf \emptyset =\sup
\mathbb{T}$ (i.e., $\sigma(M)= M$ if $\mathbb{T}$ has a maximum $M$)
and $\sup \emptyset =\inf \mathbb{T}$ (i.e., $\rho(m)= m$ if
$\mathbb{T}$ has a minimum $m$).
\end{remark}

If $\sigma(t) > t$, then we say that $t$ is right-scattered;
if $\rho(t) < t$, then $t$ is said to be left-scattered.
Points that are simultaneously right-scattered and left-scattered
are called isolated. If $t < \sup\mathbb{T}$ and $\sigma(t) = t$,
then $t$ is called right-dense; if $t >\inf \mathbb{T}$ and $\rho(t)= t$,
then $t$ is called left-dense. The graininess function
$\mu :\mathbb{T}\rightarrow [0,\infty)$ is defined by
$\mu(t) :=\sigma(t) - t$.
The derivative makes use of the set $\mathbb{T}^{\kappa}$,
which is obtained from the time scale $\mathbb{T}$ as follows:
if $\mathbb{T}$ has a left-scattered maximum $M$, then
$\mathbb{T}^{\kappa}:=\mathbb{T} \setminus \{M\}$;
otherwise, $\mathbb{T}^{\kappa}:=\mathbb{T}$.

\begin{definition} (Delta derivative \cite{AB})
Assume $f:\mathbb{T}\rightarrow \mathbb{R}$ and let
$t\in \mathbb{T}^{\kappa}$. We define
$$
f^{\Delta}(t):=\lim_{s\rightarrow t}\frac{f(\sigma(s))-f(t)}{\sigma(s)-t},
\quad t \neq \sigma(s),
$$
provided the limit exists. We call $f^{\Delta}(t)$ the delta derivative
(or Hilger derivative) of $f$ at $t$. Moreover, we say that $f$
is delta differentiable on $\mathbb{T}^{\kappa}$ provided
$f^{\Delta}(t)$ exists for all $t\in \mathbb{T}^{\kappa}$. The function
$f^{\Delta}:\mathbb{T}^{\kappa}\rightarrow \mathbb{R}$ is then called
the (delta) derivative of $f$ on $\mathbb{T}^{\kappa}$.
\end{definition}

\begin{definition} (See \cite{BP1}).
A function $f:\mathbb{T}\rightarrow \mathbb{R}$ is called
rd-continuous provided it is continuous at right-dense points in
$\mathbb{T} $ and its left-sided limits exist (finite) at left-dense
points in $\mathbb{T}$. The set of rd-continuous functions
$f:\mathbb{T}\rightarrow \mathbb{R}$ is denoted by
$\mathcal{C}_{rd}$. Similarly, a function $f:\mathbb{T}\rightarrow
\mathbb{R}$ is called ld-continuous provided it is continuous at
left-dense points in $\mathbb{T} $ and its right-sided limits exist
(finite) at right-dense points in $\mathbb{T}$. The set of
ld-continuous functions $f:\mathbb{T}\rightarrow \mathbb{R}$ is
denoted by $\mathcal{C}_{ld}$.
\end{definition}

\begin{definition} (See \cite{BP1}).
Let $[a,b]$ denote a closed bounded interval in $\mathbb{T}$. A
function $F: [a,b]\rightarrow \mathbb{R}$ is called a delta
antiderivative of function $f: [a,b)\rightarrow \mathbb{R}$ provided
$F$ is continuous on $[a,b]$, delta differentiable on $[a,b)$, and $
F^{\Delta}(t)=f(t)$ for all $t\in [a,b)$. Then, we define the
$\Delta$-integral of $f$ from $a$ to $b$ by
$$
\int_{a}^{b}f(t)\Delta t := F(b)-F(a).
$$
\end{definition}

\begin{proposition} 
\label{P1}
(See \cite{AJ}) Suppose $\mathbb{T}$ is a time scale and
$f$ is an increasing continuous function
on the time-scale interval $[a,b]$.
If $F$ is the extension of $f$ to the real interval $[a,b]$ given by
\begin{equation*}
F(s) :=
\begin{cases}
f(s) & \textrm{ if } s \in\mathbb{T} , \\
f(t) & \textrm{ if } s \in (t,\sigma(t))\notin\mathbb{T},
\end{cases}
\end{equation*}
then
$$
\int_{a}^{b} f(t) \Delta t
\leq \int_{a}^{b} F(t)dt.
$$
\end{proposition}

% ------------------------------------

\subsection{Fractional derivative and integral on time scales}

We adopt a recent notion of fractional derivative on time scales
introduced in \cite{MyID:328}, which is based on the notion
of fractional integral on time scales $\mathbb{T}$.
This is in contrast with \cite{BBT,MyID:320,MyID:324}, where first a notion
of fractional differentiation on time scales is introduced and only after that,
with the help of such concept, the fraction integral is defined.
The classical gamma and beta functions are used.

\begin{definition} (Gamma function)
For complex numbers with a positive real part,
the gamma function $\Gamma(t)$ is defined
by the following convergent improper integral:
$$
\Gamma(t) := \int_{0}^{\infty} x^{t-1} e^{-x}dx.
$$
\end{definition}

\begin{definition} (Beta function)
The beta function, also called the Euler integral of first kind,
is the special function $\mathrm{B}(x,y)$ defined by
$$
\mathrm{B}(x,y) := \int _{0}^{1}t^{x-1}(1-t)^{y-1}dt,
\quad x>0, \quad y>0.
$$
\end{definition}

\begin{remark}
The gamma function satisfies the following property:
$\Gamma(t+1)=t \Gamma(t)$.
The beta function can be expressed through the gamma function by
$$
\mathrm{B}(x,y)=\frac{\Gamma(x) \Gamma(y)}{\Gamma(x+y)}.
$$
\end{remark}

\begin{definition} 
\label{def:FI} 
(Fractional integral on time scales \cite{MyID:328})
Suppose $\mathbb{T}$ is a time scale, $[a,b]$ is an
interval of $\mathbb{T}$, and $h$ is an integrable function on
$[a,b]$. Let $0 < \alpha <1$. Then the left fractional integral of
order $\alpha$ of $h$ is defined by
$$
{_{a}^{\mathbb{T}}I}_{t}^{\alpha}h(t) := \int_{a}^{t}
\frac{(t-s)^{\alpha-1}}{\Gamma(\alpha)}h(s)\Delta s.
$$
The right fractional integral of order $\alpha$ of $h$ is defined
by
$$
{_{t}^{\mathbb{T}}I}_{b}^{\alpha}h(t) := \int_{t}^{b}
\frac{(s-t)^{\alpha-1}}{\Gamma(\alpha)}h(s)\Delta s,
$$
where $\Gamma$ is the gamma function.
\end{definition}

\begin{definition}
\label{RL} 
(Riemann--Liouville fractional derivative on time scales \cite{MyID:328}) 
Let $\mathbb{T}$ be a time scale, $t\in\mathbb{T}$, 
$0 < \alpha <1$, and $h:\mathbb{T}\rightarrow \mathbb{R}$.
The left Riemann--Liouville fractional derivative of order
$\alpha$ of $h$ is defined by
\begin{equation}
\label{eq3:L} 
{_{a}^{\mathbb{T}}D}_{t}^{\alpha}h(t)
:=\left({_{a}^{\mathbb{T}}I}_{t}^{1-\alpha}h(t)\right)^{\Delta}
=\frac{1}{\Gamma(1-\alpha)}\left(\int_{a}^{t}
(t-s)^{-\alpha}h(s)\Delta s\right)^{\Delta}.
\end{equation}
The right Riemann--Liouville fractional derivative of order
$\alpha$ of $h$ is defined by
\begin{equation*}
{_{t}^{\mathbb{T}}D}_{b}^{\alpha}h(t)
:=-\left({_{t}^{\mathbb{T}}I}_{b}^{1-\alpha}h(t)\right)^{\Delta}
=\frac{-1}{\Gamma(1-\alpha)}\left(\int_{t}^{b}
(s-t)^{-\alpha}h(s)\Delta s\right)^{\Delta}.
\end{equation*}

\end{definition}

\begin{definition}
\label{Ca} 
(Caputo fractional derivative on time scales \cite{AJ}) 
Let $\mathbb{T}$ be a time scale, $t\in\mathbb{T}$, 
$0 < \alpha <1$, and $h:\mathbb{T}\rightarrow \mathbb{R}$.
The left Caputo fractional derivative 
of order $\alpha$ of $h$ is defined by
\begin{equation*} 
{_{a}^{\mathbb{T}C}D}_{t}^{\alpha}h(t)
:={_{a}^{\mathbb{T}}I}_{t}^{1-\alpha}(h^{\Delta}(t))
=\frac{1}{\Gamma(1-\alpha)}\int_{a}^{t}
(t-s)^{-\alpha}h^{\Delta}(s)\Delta s.
\end{equation*}
The right Caputo fractional derivative of order 
$\alpha$ of $h$ is defined by
\begin{equation*}
{_{t}^{\mathbb{T}C}D}_{b}^{\alpha}h(t)
:=-{_{t}^{\mathbb{T}}I}_{b}^{1-\alpha}(h^{\Delta}(t))
=\frac{-1}{\Gamma(1-\alpha)}\int_{t}^{b}
(s-t)^{-\alpha}h^{\Delta}(s)\Delta s.
\end{equation*}
\end{definition}

The relation between the left/right RLFD and the left/right CFD
is as follows:
$$
_{a}^{\mathbb T C} D_{t}^{\alpha}x(t)=\,_{a}^{\mathbb T}
D_{t}^{\alpha}x(t)-\sum_{k=0}^{n-1}
\frac{x^{(k)}(a)}{\Gamma(k-\alpha+1)}(t-a)^{(k-\alpha)},
$$
$$
_{t}^{\mathbb T C} D_{b}^{\alpha}x(t)=\,_{t}^{\mathbb T}
D_{b}^{\alpha}x(t)-\sum_{k=0}^{n-1}
\frac{x^{(k)}(b)}{\Gamma(k-\alpha+1)}(b-t)^{(k-\alpha)}.
$$
If $x$ and $x^{(i)}$, $i = 1, \ldots , n-1$, vanish at $t = a$, 
then $_{a}^{\mathbb T}D_{t}^{\beta}x(t) =\, _{a}^{\mathbb T C}
D_{t}^{\beta}x(t)$, and if they vanish at $t = b$, then $
_{t}^{\mathbb T} D_{b}^{\beta}x(t) =\, _{t}^{\mathbb T C}
D_{b}^{\beta}x(t)$. Furthermore, 
$_{a}^{\mathbb T C}  D_{t}^{\alpha}\,c=0$, 
where $c$ is a constant, and
$$
_{a}^{\mathbb T C}  D_{t}^{\alpha}\,t^{n}
=
\begin{cases}
0,\,\, \text{for} \,\,\, n\in N_{0}\,\,\text{and}\,\,\, n<[\alpha], \\
\frac{\Gamma(n+1)}{\Gamma (n+1-\alpha)}t^{n-\alpha},\,\, 
\text{for}\,\, n\in \mathbb{N}_{0} \,\,\text{and}\, n\geq[\beta],
\end{cases}
$$
where $\mathbb{N}_{0}=\{0,1,2,\ldots\}$. 

\begin{remark} 
If $\mathbb{T}=\mathbb{R}$, then Definition~\ref{RL} gives the
classical left and right Riemann--Liouville fractional derivatives
\cite{book:Podlubny}. Similar comment for Definition~\ref{Ca}.
For different extensions of the fractional
derivative to time scales using the Caputo approach, 
see \cite{MR2800417}. For local approaches to
fractional calculus on time scales, we refer the reader to
\cite{BBT,MyID:320,MyID:324}.  Here we restrict ourselves to the
delta approach to time scales. Analogous definitions are, however,
trivially obtained for the nabla approach to time scales by using
the duality theory of \cite{MR3662975,MyID:307}.
\end{remark}

% -------------------

\subsection{Properties of the time-scale fractional operators}
\label{sec:properties}

We recall some fundamental properties
of the fractional operators on time scales.

\begin{proposition}
\label{P2} (See Proposition~15 of \cite{MyID:328}).  
Let $\mathbb{T}$ be a time scale with derivative $\Delta$, 
and $0 < \alpha <1$. Then,
${_{a}^{\mathbb{T}}D}_{t}^{\alpha}
= \Delta \circ {_{a}^{\mathbb{T}}I_{t}^{1-\alpha}}$.
\end{proposition}

\begin{proposition}
\label{P3} (See Proposition~16 of \cite{MyID:328}). 
For any function $h$ integrable on $[a,b]$, the Riemann--Liouville
$\Delta$-fractional integral satisfies
${_{a}^{\mathbb{T}}I}_{t}^{\alpha} \circ
{_{a}^{\mathbb{T}}I_{t}^{\beta}} =
{_{a}^{\mathbb{T}}I_{t}^{\alpha+\beta}}$ 
for $\alpha> 0$ and $\beta > 0$.
\end{proposition}

\begin{proposition}
\label{P4} (See Proposition~17 of \cite{MyID:328}). For any function
$h$ integrable on $[a,b]$ one has ${_{a}^{\mathbb{T}}D}_{t}^{\alpha}
\circ {_{a}^{\mathbb{T}}I_{t}^{\alpha}}h = h$.
\end{proposition}

\begin{corollary} (See Corollary~18 of \cite{MyID:328}).
For $0<\alpha<1$, we have ${_{a}^{\mathbb{T}}D}_{t}^{\alpha} \circ
{_{a}^{\mathbb{T}}D_{t}^{-\alpha}} =Id$ and
${_{a}^{\mathbb{T}}I}_{t}^{-\alpha} \circ
{_{a}^{\mathbb{T}}I_{t}^{\alpha}}=Id$, where $Id$ denotes the
identity operator.
\end{corollary}

\begin{definition}(See \cite{MyID:328})
For $\alpha>0$, we denote by ${_{a}^{\mathbb{T}}I}_{t}^{\alpha}([a,b])$
the space of functions that can be represented by the
Riemann--Liouville $\Delta$ integral of order $\alpha$ 
of some $\mathcal{C}([a,b])$-function.
\end{definition}

\begin{theorem}
\label{th1} (See Theorem~20 of \cite{MyID:328}). Let $f\in
\mathcal{C}([a,b])$ and $\alpha>0$. Function $f \in
{_{a}^{\mathbb{T}}}I_{t}^{\alpha}([a,b])$ if and
only if 
${_{a}^{\mathbb{T}}I}_{t}^{1-\alpha}f \in C^1([a,b])$
and
$\left.\left(_{a}^{\mathbb{T}}I_{t}^{1-\alpha}f(t)\right)\right|_{t=a}=0$.
\end{theorem}

\begin{theorem}
\label{th2}
(See Theorem~21 of \cite{MyID:328}) Let $\alpha > 0$ and
$f\in \mathcal{C}([a,b])$ satisfy the conditions in
Theorem~\ref{th1}. Then,
$\left({_{a}^{\mathbb{T}}I}_{t}^{\alpha}
\circ {_{a}^{\mathbb{T}}D}_{t}^{\alpha}\right)(f) = f$.
\end{theorem}

% -------------------

\subsection{Existence of solutions to fractional IVPs on time scales}
\label{sec:existence:sol}

Let $\mathbb{T}$ be a time scale and
$\mathcal{J}=[t_{0},t_{0}+a]\subset \mathbb{T}$.
Consider the fractional order initial value problem \eqref{eq1}
defined on $\mathbb{T}$. Then the function 
$y\in\mathcal{C}(\mathcal{J},\mathbb{R})$ is a solution
of problem \eqref{eq1} if
${_{t_{0}}^{\mathbb{T}}D}_{t}^{\alpha}y(t)= f(t,y)$
on $\mathcal{J}$ and
${_{t_{0}}^{\mathbb{T}}I}_{t}^{\alpha}y(t_{0})=0$.

\begin{theorem}(See Theorem~24 of \cite{MyID:328}).
If $f :\mathcal{J}\times\mathbb{R}\rightarrow\mathbb{R}$ is a
rd-continuous bounded function for which there exists $M > 0$ such that
$|f(t,y)|\leq M$ for all $t\in\mathcal{J}$ and $y\in\mathbb{R}$, 
then problem~\eqref{eq1} has a solution on $\mathcal{J}$.
\end{theorem}

% -------------------------

\section{Main Results}

We begin by proving formulas of integration by parts 
in Section~\ref{sec:MR:01}, which are then used in
Section~\ref{sec:MR:02} to prove necessary optimality
conditions for nonlinear Riemann--Liouville fractional
optimal control problems (FOCPs) on time scales.

% ----------------------

\subsection{Fractional integration by parts on time scales}
\label{sec:MR:01}

Our first result gives integration by parts formulas 
for fractional integrals and derivatives on time scales.
For the relation between integration on time scales 
and Lebesgue integration we refer the reader to
\cite{MR2206702}.

\begin{theorem}
\label{l1}
Let $\alpha > 0$, $p,q\geq 1$, and $\frac{1}{p}+ \frac{1}{q} \leq
1+\alpha$, where $p \ne 1 $ and $q \ne 1$ in the case when 
$\frac{1}{p}+\frac{1}{q} = 1+\alpha$. Moreover, let
\begin{equation*}
{_{a}^{\mathbb{T}}I}_{t}^{\alpha}(L_p)
:=\left\{f: f={_{a}^{\mathbb{T}}I}_{t}^{\alpha}g, \,
g\in L_p(a,b)\right\} 
\end{equation*}
and
\begin{equation*}
{_{t}^{\mathbb{T}}I}_{b}^{\alpha}(L_p)
:=\left\{f: f={_{t}^{\mathbb{T}}I}_{b}^{\alpha}g, \,
g\in L_p(a,b)\right\}.
\end{equation*}
The following integration by parts formulas hold.
\begin{itemize}
\item[(a)]\  
If $\varphi \in L_p(a,b)$ and  $\psi \in L_q(a,b)$, then
\begin{equation} 
\label{e1}
\int_{a}^{b} \varphi(t)\left({_{a}^{\mathbb{T}}I}_{t}^{\alpha}
\psi\right)(t)\Delta t =\int_{a}^{b}
\psi(t)\left({_{t}^{\mathbb{T}}I}_{b}^{\alpha} \varphi\right)(t)\Delta t.
\end{equation}
\item[(b)] \ 
If $g \in {_{t}^{\mathbb{T}}I}_{b}^{\alpha}(L_p)$ and $f \in
{_{a}^{\mathbb{T}}I}_{t}^{\alpha}(L_q)$, then
\begin{equation} 
\label{d1}
\int_{a}^{b}g(t) \left({_{a}^{\mathbb{T}}D}_{t}^{\alpha} f\right)(t)\Delta t 
= \int_{a}^{b} f(t) \left({_{t}^{\mathbb{T}}D}_{b}^{\alpha} g\right)(t)\Delta t.
\end{equation}
\item[(c)]\  For Caputo fractional derivatives,
if $g \in {_{t}^{\mathbb{T}}I}_{b}^{\alpha}(L_p)$ and 
$f \in {_{a}^{\mathbb{T}}I}_{t}^{\alpha}(L_q) $, then
\begin{equation*}
\int_{a}^{b} g(t) \left( 
{_{a}^{\mathbb{T}C}D}_{t}^{\alpha} f\right)(t)\Delta t
= \left[{_{t}^{\mathbb{T}}I}_{b}^{1-\alpha} g(t)\cdot f(t) \right]_{a}^{b}
+ \int_{a}^{b}f(\sigma(t)) \left({_{t}^{\mathbb{T}}D}_{b}^{\alpha} g\right)(t)
\Delta t
\end{equation*}
and
\begin{equation*}
\int_{a}^{b}g(t) \left( 
{_{t}^{\mathbb{T}C}D}_{b}^{\alpha} f\right)(t)\Delta t
= - \left[ {_{a}^{\mathbb{T}}I}_{t}^{1-\alpha} g(t)\cdot f(t) \right]_{a}^{b}
+ \int_{a}^{b}f(\sigma(t)) \left({_{a}^{\mathbb{T}}D}_{t}^{\alpha} g\right)(t)\Delta t.
\end{equation*}
\end{itemize}
\end{theorem}

\begin{proof}
(a) If $\varphi \in L_p(a,b)$ and  $\psi \in L_q(a,b)$,
then, from Definition~\ref{def:FI}, we get
$$
\int_{a}^{b} \varphi(t)\left({_{a}^{\mathbb{T}}I}_{t}^{\alpha}
\psi\right)(t)\Delta t =\int_{a}^{b} \varphi(t)\left(\int_{a}^{t}
\frac{(t-s)^{\alpha-1}}{\Gamma(\alpha)}\psi(s)\Delta s\right)\Delta t.
$$
Interchanging the order of integrals (see \cite{MyID:328}), we reach at
$$
\int_{a}^{b} \varphi(t)\left({_{a}^{\mathbb{T}}I}_{t}^{\alpha}
\psi\right)(t)\Delta t =\int_{a}^{b}
\psi(t)\left({_{t}^{\mathbb{T}}I}_{b}^{\alpha} \varphi\right)(t)\Delta t.
$$
(b) If $g \in {_{t}^{\mathbb{T}}I}_{b}^{\alpha}(L_p)$ and 
$f \in {_{a}^{\mathbb{T}}I}_{t}^{\alpha}(L_q)$, then, 
from Definition~\ref{RL}, we get
$$
\int_{a}^{b}g(t) \left( 
{_{a}^{\mathbb{T}}D}_{t}^{\alpha} f\right)(t)\Delta t
=\int_{a}^{b}g(t) \left(\frac{1}{\Gamma(1-\alpha)}\left(\int_{a}^{t}
(t-s)^{-\alpha}f(s)\Delta s\right)^{\Delta} \right)\Delta t.
$$
Interchanging the order of integrals, we obtain that
$$
\int_{a}^{b}g(t) \left( {_{a}^{\mathbb{T}}D}_{t}^{\alpha} f\right)(t)\Delta t
=\int_{a}^{b}f(t) \left({_{t}^{\mathbb{T}}D}_{b}^{\alpha} g\right)(t)\Delta t.
$$
(c) If $g \in {_{t}^{\mathbb{T}}I}_{b}^{\alpha}(L_p)$ 
and $f \in {_{a}^{\mathbb{T}}I}_{t}^{\alpha}(L_q)$, 
then, from Definition~\ref{Ca}, we get
$$
\int_{a}^{b}g(t) \left({_{a}^{\mathbb{T}C}D}_{t}^{\alpha} f\right)(t)
\Delta t=\int_{a}^{b}g(t) \left(\frac{1}{\Gamma(1-\alpha)}\int_{a}^{t}
(t-s)^{-\alpha}f^{\Delta}(s)\Delta s \right)\Delta t.
$$
Interchanging the order of the integrals, 
and by using integration by parts on time scales, 
we conclude that
$$
\int_{a}^{b}g(t) \left({_{a}^{\mathbb{T}C}D}_{t}^{\alpha} f\right)(t)\Delta t
= \int_{a}^{b}f(\sigma(t)) \left({_{t}^{\mathbb{T}}D}_{b}^{\alpha} g\right)(t)\Delta t
+\left[ {_{t}^{\mathbb{T}}I}_{b}^{1-\alpha} g(t)\cdot f(t)\right]_{a}^{b}.
$$
The second relation is obtained in a similar way.
\end{proof}

% ----------------------------------

\subsection{Nonlinear Riemann--Liouville FOCPs on time scales}
\label{sec:MR:02}

Let $\mathbb{T}$ be a given time scale with $t_0, t_f \in \mathbb{T}$
and let us consider a control system given by the fractional  
differential equation
\begin{equation}
\label{4.1}
^{\mathbb{T}}_{\,t_0}D_{t}^{\alpha}x(t)=f(x(t),u(t),t),
\quad t\in {\mathbb T},
\end{equation}
subject to
\begin{equation}
\label{4.2}
^{\mathbb{T}}_{\, t_0}I_{t}^{1-\alpha}x(t_0)=x_0,
\end{equation}
where $x(t)\in \mathbb{R}^n$ and $u(t)\in \mathbb{R}^m$ 
are the state and control vectors, respectively, function
$f:\mathbb{R}^{n\times m\times 1}\rightarrow \mathbb{R}^{n}$
is a nonlinear vector function, and $x_0\in \mathbb{R}^n$ 
is the specified initial state vector. A similar problem 
is studied in \cite{MR2163462} for problems
involving AB derivatives in Caputo sense on continuous time scales. 
Here we study it within Riemann--Liouville derivatives on arbitrary 
time scales. In order to achieve a desired behavior in terms of
performance requirements, we select a cost index for
the dynamical system \eqref{4.1}--\eqref{4.2}. In selecting the
performance index, the designer attempts to define a mathematical
expression that, when minimized, indicates that the system is
performing in the most desirable manner. Thus, choosing a
performance cost index is a translation of system's physical requirements
into mathematical terms \cite{Bal.4}. For the fractional dynamic system
\eqref{4.1}--\eqref{4.2}, we choose the following performance index:
\begin{equation}
\label{4.3} 
J[x,u]=\int_{t_0}^{t_{f}} L(x(t),u(t),t) \Delta t\longrightarrow \min,
\end{equation} 
where  $L:\mathbb{R}^{n\times m\times 1}\to \mathbb{R}$ is a scalar function.
In the following, we derive a necessary optimality condition
corresponding to the considered fractional optimal control problem
\eqref{4.1}--\eqref{4.3}. Under given considerations, 
the following theorem holds true.

\begin{theorem}
\label{thm4.1}
(Necessary optimality conditions) 
Let $(x(\cdot), u(\cdot))$ be a minimizer of 
problem \eqref{4.1}--\eqref{4.3}. Then, there exists 
a function $\lambda(\cdot)$ for which the triplet 
$\left(x(\cdot), \lambda(\cdot), u(\cdot)\right)$ satisfies:
\begin{itemize}
\item[(i)] \ \emph{the Hamiltonian system}
\begin{equation}
\label{4.4}
\begin{cases}
^{\mathbb{T}}_{\,t_{0}}D_{t}^{\alpha}x(t)
=\displaystyle \frac{\partial \mathcal H}{\partial \lambda}\left(x(t), 
\lambda (t), u(t), t\right), & t\in {\mathbb T},\\[0.3cm]
^{\mathbb{T}}_{\,t}D_{t_f}^{\alpha}\lambda(t)
=\displaystyle \frac{\partial \mathcal H}{\partial x}\left(x(t), 
\lambda (t), u(t), t\right), & t\in {\mathbb T};
\end{cases}
\end{equation}
\item[(ii)] \ \emph{the stationary condition}
\begin{equation}
\label{4.6}
\frac{\partial \mathcal H}{\partial u}\left(x(t), \lambda (t), u(t), t\right)
=0,\quad t\in {\mathbb T},
\end{equation}
where ${\mathcal H}$ is a scalar function, called the \emph{Hamiltonian},
defined by
\begin{equation}
\label{4.7}
{\mathcal H}(x,\lambda, u, t)=L(x, u, t)
+\lambda^{T} f(x, u, t).
\end{equation}
\end{itemize}
\end{theorem}

\begin{proof}
To deduce the necessary optimality conditions that the optimal pair
$\left(x(\cdot), u(\cdot)\right)$ must satisfy, we use the Lagrange 
multiplier technique to adjoin the dynamic constraint \eqref{4.1} 
to the performance index \eqref{4.3}. Thus, we form the augmented 
functional
\begin{equation}
\label{4.8}
J_{a}[x,\lambda,u]=\int_{t_0}^{t_{f}} \left[{\mathcal H}(x(t),\lambda (t), u(t),
t)-\lambda^{T}(t)^{\mathbb{T}}_{\,t_{0}}D_{t}^{\alpha}x(t)\right] \Delta t,
\end{equation}
where $\lambda(t)\in \mathbb{R}^{n}$ is the Lagrange multiplier, also
known as the costate or adjoint variable. Taking the first variation 
of the augmented performance index $J_{a}[x,\lambda,u]$ given by \eqref{4.8}, 
we obtain that
\begin{equation}
\label{4.9}
\begin{split} 
\delta J_{a}[x,\lambda,u]
&=\int_{t_{0}}^{t_f}\Biggl\{\left[\frac{\partial
{\mathcal H}}{\partial x}\right]^{T}\delta x(t)
+\left[\frac{\partial {\mathcal H}}{\partial
\lambda}-{^{\mathbb{T}}_{\,t_{0}}}D_{t}^{\alpha}x(t)\right]^{T}\delta\lambda(t)\\
&\qquad +\biggl[\frac{\partial {\mathcal H}}{\partial u}\biggr]^{T}\delta u(t)
- \lambda^{T}(t)^{\mathbb{T}}_{\,t_{0}}D_{t}^{\alpha}\delta x(t)\Biggr\}\Delta t.
\end{split}
\end{equation}
Using the fractional integration by parts formula \eqref{d1}, 
the last integral in \eqref{4.9} can be written as
\begin{equation}
\label{4.10}
\int_{t_{0}}^{t_f}\lambda^{T}(t)\,
^{\mathbb{T}}_{\,t_{0}}D_{t}^{\alpha}\delta x(t)\Delta
t=\int_{t_{0}}^{t_f}\biggl(\,
^{\mathbb{T}}_{\,t}D_{t_f}^{\alpha}\lambda(t)\biggr)^T\delta
x(t)\Delta t.
\end{equation}
Using \eqref{4.10} in \eqref{4.9}, we deduce that
\begin{equation}
\label{4.11}
\begin{split} \delta J_{a}[x,\lambda,u]
&=\int_{t_{0}}^{t_f}\Biggl\{\left[\frac{\partial
{\mathcal H}}{\partial x}-\,
^{\mathbb{T}}_{\,t}D_{t_f}^{\alpha}\lambda(t)\right]^{T}\delta x(t)
+\left[\frac{\partial {\mathcal H}}{\partial
\lambda}-{^{\mathbb{T}}_{\,t_{0}}}D_{t}^{\alpha}x(t)\right]^{T}\delta\lambda(t)\\
&\qquad +\left[\frac{\partial {\mathcal H}}{\partial u}\right]^{T}\delta u(t)\Biggr\}\Delta t.
\end{split}
\end{equation}
The necessary condition for an extremal asserts that the first variation
of $J_{a}[x,\lambda,u]$ must vanish along the extremal for all independent
variations $\delta x(t)$, $\delta\lambda(t)$, and $\delta u(t)$. Because
of this, all factors multiplying a variation in Eq. \eqref{4.11}
must vanish. We obtain conditions \eqref{4.4}--\eqref{4.6}. 
\end{proof}

Equations \eqref{4.4}--\eqref{4.6} represent the Euler--Lagrange 
equations of the FOCP \eqref{4.1}--\eqref{4.3}. Note that
Theorem~\ref{thm4.1} covers fractional optimal control problems
defined on isolated time scales with a non-constant graininess, 
as well as variational problems on time scales that are partially 
continuous and partially discrete, i.e., on hybrid time scales.

% ----------------------------------------------

\subsection{An illustrative example}
\label{sec:ex}

Let $\mathbb{T}$ be a time scale with $0, T \in \mathbb{T}$. 
Consider the control system 
\begin{equation}
\label{4.12}
^{\mathbb{T}}_{\,0}D_{t}^{\alpha}x(t)= u(t),
\quad t\in [0,T]_\mathbb{T},
\end{equation}
subject to the initial condition
\begin{equation}
\label{4.13}
^{\mathbb{T}}_{0}I_{t}^{1-\alpha}x(0)=x_0,
\end{equation}
where the control $u$ belongs to $L^{2}$. 
Consider the problem of minimizing
$$
J[x,u]=\frac{1}{2} \left( ||x-z||_{L^2}^{2}+N||u||_{L^2}^{2} \right) 
$$
subject to \eqref{4.12}--\eqref{4.13}, 
where $z\in L^2$ and $N>0$ are fixed/given. 
In agreement with Theorem~\ref{thm4.1}, the optimal control
$u$ is characterized by \eqref{4.12}--\eqref{4.13} with
the adjoint system
$$
^{\mathbb{T}}_{\,t}D_{T}^{\alpha}\lambda(t)=x(t)-z(t),
\quad t\in \, [0,T]_\mathbb{T},
$$
and with the optimality condition
$$
u(t)=-\frac{\lambda(t)}{N}.
$$

% ----------------------------------------------

\section{Conclusion}
\label{sec:conc}

We studied optimal control problems for fractional initial values 
systems involving fractional-time derivatives on time scales. 
As a main result, a necessary optimality condition is proved. 
In the formulation of the optimal control problem, 
the control $u$ takes values in $\mathbb{R}^m$. As future work,
it would be interesting to consider the case where the control 
takes values on a closed subset of $\mathbb{R}^m$. This question 
is far from being trivial \cite{MR3740678,MR3669665}
and needs further developments. 

% ----------------------------------------------

\begin{acknowledgement}
Torres has been partially supported by
\emph{Funda\c{c}\~{a}o para a Ci\^{e}ncia e a Tecnologia} (FCT)
through CIDMA, project UID/MAT/04106/2019. The authors are
grateful to two anonymous referees for several pertinent
questions and comments.
\end{acknowledgement}

% ----------------------------------------------

% ------------------------

\end{document}